\documentclass[twoside,a4paper,reqno,12pt]{amsart} 
\usepackage{fullpage}
\usepackage{framed}
\usepackage{amsfonts, amsmath, amssymb, mathrsfs, bm, latexsym, stmaryrd, array, hyperref, mathtools, bold-extra, xcolor}
\usepackage{booktabs}

 \title{On the diameter of intersection graphs of finite groups}

\author{Melissa Lee} 
\address{School of Mathematics, Monash University, Clayton VIC 3800, Australia}
\email{melissa.lee@monash.edu}

\author{Kamilla Rekv\'enyi} 
\address{Imperial College London, 180 Queen's Gate, South Kensington, London, SW7 2AZ, United Kingdom}
\email{k.rekvenyi19@imperial.ac.uk}
\date{February 2024}
\newtheorem{theorem}{Theorem} 
\newtheorem*{conj*}{Conjecture}
\newtheorem{lemma}[theorem]{Lemma}

\newtheorem{thm}{Theorem}[section]

\theoremstyle{definition}
\newtheorem{rem}[thm]{Remark}

\begin{document}
\begin{abstract}
The intersection graph $\Delta_G$ of a finite group $G$ is a simple graph with vertices the non-trivial proper subgroups of $G$, and an edge between two vertices if their corresponding subgroups intersect non-trivially. These graphs were introduced by Cs\'ak\'any and Poll\'ak in 1969. In this paper we answer two long-standing open questions posed by Cs\'ak\'any and Poll\'ak concerning the diameter of intersection graphs. We prove some necessary conditions for a non-simple group to have an intersection graph of diameter 4. We also construct the first examples of non-simple groups and alternating groups whose intersection graphs have diameter 4.
\end{abstract}

\maketitle
Let $G$ be a finite group. The intersection graph $\Delta_G$ of $G$ is a simple graph with vertices the non-trivial proper subgroups of $G$, and an edge between two vertices if the corresponding subgroups intersect non-trivially. These graphs were introduced by Cs\'ak\'any and Poll\'ak \cite{csakany} in 1969, influenced by earlier work by Bos\'ak \cite{bosak} on analogous graphs for semigroups.
In their paper, Cs\'ak\'any and Poll\'ak prove that a finite non-simple group with disconnected intersection graph either has trivial centre and every proper subgroup is abelian, or it is a direct product of two cyclic groups of prime order. They further show that the intersection graph of every other 
non-simple finite group has $\mathrm{diam}(\Delta_G)\leq 4$ with equality only if $G$ is an almost simple group of the form $G = S\rtimes C_p$ for some non-abelian finite simple group and prime $p$. Freedman further showed that in fact, $p$ must be an odd prime \cite[p.83]{saulthesis}. Despite this reduction, no example of a non-simple group with intersection graph of diameter 4 has been found to date. We change this, by presenting a group whose intersection graph has diameter 4, as well as providing some necessary conditions for any further examples.

\begin{theorem}
\label{thm:nonsimple}
Suppose $G$ is a non-simple group with connected intersection graph $\Delta_G$. Then $\mathrm{diam}(\Delta_G) \leq 4$ with equality only if $G$ is almost simple with socle $G_0$, $G=\langle G_0, g\rangle$, where $g$ is a diagonal automorphism of odd prime order $p$, and one of the following holds.
\begin{enumerate}
    \item $G_0 = \mathrm{PSL}_n(q)$ with $n=p$ prime, and $n\mid q- 1$;
    \item $G_0 = \mathrm{PSU}_n(q)$ with $p\mid (n,q+1)$, and either $n=p$ prime, or $n=2p$;
    \item $G_0 = E_6(q)$ or ${}^2E_6(q)$ with $q \equiv \pm 1 \mod 3$ respectively and $p=3$.
\end{enumerate}
Moreover, if $G=\mathrm{PGU}_5(4)$, then $\mathrm{diam}(\Delta_G) = 4$.
\end{theorem}

It should be noted that the conditions of Theorem \ref{thm:nonsimple} are necessary but not sufficient. For example, $G= \mathrm{PGU}_3(5)$ has an intersection graph of diameter 3.

Similar questions concerning the diameter of intersection graphs of non-abelian finite simple groups have also been considered. In 2010, Shen \cite{shen} proved that here the intersection graph is always connected and asked whether there is an upper bound for the diameter that holds for all intersection graphs of finite simple groups. In \cite{herzog}, Herzog, Longobardi and Maj showed that for a finite simple group $G,$    $\mathrm{diam}(\Delta_G) \leq 64.$ This bound was later reduced to $28$ by Ma in \cite{ma} and to $5$ by Freedman in \cite{saul}. The latter result is shown to be sharp with equality for $G=\mathbb{B},$ $\mathrm{PSU}_7(2)$ and possibly further $\mathrm{PSU}_n(q)$ examples with $n$ an odd prime \cite{saul}. In particular, in all other cases, $\mathrm{diam}(\Delta_G) \leq 4$. In \cite{csakany}, Csákány and Pollák showed that for alternating groups $A_n$, we have $3\leq \mathrm{diam}(\Delta_{A_n})\leq 4$, however it has been unclear whether the upper bound is sharp. We prove that it is.  

\begin{theorem}
    \label{thm:An}
    Let $G=A_n$, where $n$ is a prime not equal to 11 or $(q^d-1)/q-1$ for some prime power $q$ and positive integer $d$. The intersection graph of $G$ has diameter 4.
\end{theorem}

The smallest example of a simple alternating group with an intersection graph of diameter 4 is $A_{13}$, while all smaller alternating groups have intersection graphs of diameter 3.
% In this paper we consider the long standing open question whether there is a non-simple group whose intersection graph has diameter 4. We prove that the answer is no.  We also give an example of a family of simple groups whose intersection graph has diameter 4.

The remainder of the paper is divided into two sections. The first deals with the proof of Theorem \ref{thm:nonsimple}, and the second with Theorem \ref{thm:An}.
\subsection*{Acknowledgements} The first author acknowledges the support of an Australian Research Council Discovery Early Career Researcher Award (project number DE230100579). 
The second author was supported by the UK Engineering and Physical Sciences Research Council [grant number EP/W522673/1]. The authors would also like to thank Saul Freedman for introducing them to this problem, as well as Aluna Rizzoli for insightful conversations during the completion of this paper. 
The authors are also grateful to Martin Liebeck for carefully proofreading an earlier draft and highlighting opportunities to simplify some of the arguments.
\section{Non-simple groups with connected intersection graphs}
In this section we present our result on the intersection graphs of non-simple groups. As discussed in the introduction, our result settles a long-standing open question of Csákány and Pollák whether there exists a finite non-simple group whose intersection graph has diameter $4.$ We prove that such groups exist, and provide some necessary conditions for any further examples.

We first begin with some preliminary results. Note that this result is also in \cite{csakany}.
\begin{lemma}
\label{lem:diam}
Suppose $G$ is a finite group with a connected intersection graph with diameter $d\geq 1$. Then there exist cyclic subgroups of prime order $A,B < G$ at distance $d$.
\end{lemma}
\begin{proof}
Suppose $H_1, H_2 < G$ are subgroups at distance $d$ and let $A$, $B$ be cyclic subgroups of $G$ of prime order with $A\leq H_1$ and $B\leq H_2$. Every subgroup of $G$ intersecting $A$ also intersects $H_1$, and every subgroup of $G$ intersecting $B$ also intersects $H_2$.  Hence if  $A_0$, $A_1, \dots ,A_k$ is a path with $A_0=A$ and $A_k=B,$ then  $H_1, A_1,\dots, A_{k-1},H_2$ is also path. Therefore, the distance between $H_1$ and $H_2$ provides a lower bound for the distance between $A$, $B,$ which, as the diameter is $d,$  is also equal to $d$.  
\end{proof}

We start by describing the groups $G$ such that $\Delta_G$ has diameter $2.$ Note that one direction of this result is also in \cite{ravi}.

\begin{lemma}
    Let $G$ be a finite group. Then $\mathrm{diam}(\Delta_G)=2$ if and only if $G$ does not have a generating pair of prime order elements.
\end{lemma}
\begin{proof}
If $\mathrm{diam}(\Delta_G)=2,$ then for any two distinct cyclic subgroups of prime order, $A=\langle g_a\rangle $ and  $B=\langle g_b\rangle,$ there is a proper subgroup that contains both of them. The smallest such subgroup is $\langle g_a,g_b\rangle,$ which therefore is not equal to $G,$ so no pairs of prime order elements generate $G.$ 

    If $G$ does not have a generating pair of prime order elements, then for any two prime order elements $g_a,$ $g_b,$ the subgroup they generate, $\langle g_a,g_b\rangle,$ is a proper subgroup of $G.$ Hence there is a path of length two between $A=\langle g_a\rangle $ and  $B=\langle g_b\rangle,$ in $\Delta_G$, so by Lemma \ref{lem:diam},  $\mathrm{diam}(\Delta_G)=2.$
\end{proof}

% \begin{theorem}
%     Let $G$ be a finite, non-simple group such that $\Delta_G$ is connected. Then $\mathrm{diam}(\Delta_G)=3.$
% \end{theorem}

The argument in the proof of Theorem \ref{thm:nonsimple} relies heavily on outer automorphisms normalising certain large parabolic subgroups of groups of Lie type. For this, we require the following lemma.
%before embarking on the proof of Theorem \ref{thm:nonsimple}.

\begin{lemma}\label{nondiagaut}
    Let $S$ be a finite non-abelian simple group of Lie type. Then any non-diagonal outer automorphism of odd prime order of $S$ normalizes a parabolic maximal subgroup of $S$. 
    Moreover, any two of these automorphisms normalise parabolic maximal subgroups that are conjugate. 
\end{lemma}

\begin{proof}
First note that $\mathrm{Aut}(S)$ is a split extension of $\mathrm{InnDiag}(S)$ by a group $\Phi_S\Gamma_S,$ where $\Phi_S$ and $\Gamma_S$ contain field and graph automorphisms, respectively. By \cite[Thm. 2.5.12, Defn. 2.5.13]{gls}, any non-diagonal outer automorphism $g$ of $S$ of odd prime order is a field automorphism, or $S\cong \mathrm{P}\Omega^+_8(q)$ or ${}^3D_4(q)$ and $g$ is a graph or graph-field automorphism. In all cases, $g$ is conjugate to an element of $\Phi_S\Gamma_S$.

Now we show that for any two non-diagonal outer automorphisms $g_1$ and $g_2$ there is a parabolic subgroup $P$ such that both $g_1$ and $g_2$ normalize a conjugate of $P$.

If $g_1$ and $g_2$ are field automorphisms, then, in particular, they are each conjugate to an element of $\Phi_S$ by \cite[Defn. 2.5.13]{gls}. Now $\Phi_S$ normalises all parabolic subgroups of $S$ by \cite[\S2.6]{gls}, so the claim follows.
If $g_1$ and $g_2$ are graph or graph-field automorphisms, then $S\cong \mathrm{P}\Omega^+_8(q)$ or ${}^3D_4(q)$ and the parabolic maximal subgroup $P_2$ of $S$ is normalized by $\Gamma_S.$ % by \cite[Table 8.50]{colva}.
Hence $P_2$ is normalized by $\Phi_S\Gamma_S,$ and as all non-diagonal automorphisms are conjugate to an element of $\Phi_S\Gamma_S,$ $g_1$ and $g_2$ both normalize a conjugate of $P_2,$ as required.
%From \cite[Proposition 2.6.9]{gls} we know that the normalizer of a parabolic subgroup, $P_J,$ corresponding to a subset $J$ of the base of the root system of $S,$ is normalised by $\Phi_S\Gamma_J,$ where  $\Gamma_J$ is the subgroup of $\Gamma_S$ that stabilizes $J.$ 
%First note that by inspecting the Dynkin diagram corresponding to each group of Lie type, we can see that there is always a subset $D$ of the base of the root system of $S$ that is fixed by all elements of $\Gamma_S.$ From \cite[Proposition 2.6.9]{gls} we know that the normalizer of a parabolic subgroup, $P_J,$ corresponding to a subset $J$ of the base of the root system of $S,$ is normalised by $\Phi_S\Gamma_J,$ where  $\Gamma_J$ is the subgroup of $\Gamma_S$ that stabilizes $J.$ \textbf{(discuss)} Now $\Gamma_D=\Gamma_S,$ so the parabolic $P_D$ is normalised by $\Phi_S\Gamma_S,$ and as all non-diagonal outer automorphisms of $S$ are conjugate to an element of $\Phi_S\Gamma_S,$ they all normalize a conjugate of $P_D.$
%Since by \cite[Def 2.5.13]{gls} and \cite[3D4??]{colva} all field, graph and graph-field automorphisms of $S$ are conjugate to an element in $\Phi_S\Gamma_S,$  this shows that they all normalize a conjugate of a parabolic for some $J.$ 
\end{proof}

% \begin{lemma}
%     Let $S$ be a finite non-abelian simple group of Lie type, and let $g \in \mathrm{InnDiag}(S)$ be a diagonal automorphism of odd prime order. Then one of the following holds.
%     \begin{enumerate}
%         \item $g$ is contained in a maximal parabolic subgroup of $\mathrm{InnDiag}(S)$.
%         \item $g$ 
%     \item \label{i2} $S = \mathrm{PSL}_n(q)$ or $\mathrm{PSU}_n(q)$ with $n=r$ prime, and $n\mid q\mp 1$ respectively, or
%     \item \label{i3} $G_0 = E_6(q)$ or ${}^2E_6(q)$ with $q \equiv \pm 1 \mod 3$ respectively and $r=3$.
% \end{enumerate}
% \end{lemma}
% \begin{proof}
% Since $g$ has odd prime order, it follows from \cite[Table 5]{atlas} that either $S$ is as in parts \ref{i2} or \ref{i3} of the lemma, or $S = \mathrm{PSL}_n(q)$ or $\mathrm{PSU}_n(q)$ with $n$ composite. So suppose the latter case. By \cite[Rmk 3.2.3]{bg}

% \end{proof}
\begin{rem} By \cite[Prop. 2.6.9]{gls} (see also \cite[Chap 8.]{colva} and \cite[Tables 3.5A--G]{KL}), an analogue of Lemma \ref{nondiagaut} holds for diagonal automorphisms that are conjugate to the so-called ``standard" diagonal automorphisms (cf. \cite[\S1.7.1]{colva}). However, not every diagonal automorphism is conjugate to a standard one, and therefore not always contained in a parabolic subgroup, as we shall see in the proof of Theorem \ref{thm:nonsimple}.
\end{rem}
\subsection{Proof of Theorem \ref{thm:nonsimple}}
\begin{proof}
    %By \cite{ravi} in this case the diameter is at least 3, so it suffices to prove that $\mathrm{diam}(\Delta_G)\leq 3.$ 
    By \cite{csakany}, if $\mathrm{diam}(\Delta_G)=4,$ then $G=S\rtimes C_p,$ where $S$ is a non-abelian simple group and $p$ is an odd prime.  

Hence $S$ cannot be a sporadic group or an alternating group with $n=5$ or $n\geq 7$, since the outer automorphism group has size at most two in each case. If $S=A_6$, then $|\mathrm{Out}(S)| = 4$, so again there is no suitable $G$.
%This implies  $S=X_r(l),$ where $X_r(l)$ is a finite simple group of Lie type of Lie rank $r$ over the field of $l$ elements.
Therefore, $S$ is a group of Lie type. Write $G = S\rtimes \langle g \rangle$, where $g$ has prime order $p$. 

Suppose that $A$ and $B$ are subgroups at distance 4 in $\Delta_G$. By Lemma \ref{lem:diam}, we can assume that $A$, $B$ are cyclic of prime order.
First suppose that $A$ or $B$ is a subgroup of $S$.
Without loss, we can assume $A\leq S.$ %By the argument above, we either have $B\leq S$, or $B< M\rtimes C_p$ for some parabolic subgroup $M$ of $S$.
 Since the graph is connected, $B$ is contained in a maximal subgroup of $G$ that has size strictly larger than $\vert B \vert,$ so either $B$ is contained in a conjugate of $S$ or $B$ has order $p$ and is properly contained in a maximal subgroup $M_B$ of $G$ with $|M_B| >p.$ If $M_B$ and $S$ had a non-trivial intersection then there would be a path between A and B of length 3, contradicting our assumption that their distance is 4. Hence we must have $|S\cap M_B| = 1$. But then $|SM_B| = |S||M_B|> p|S| = |G|$, a contradiction. 
 Hence we can assume that neither $A$ nor $B$ are contained in a conjugate of $S$ and hence they both have order $p.$ We have several cases to consider.

 The first case is when $G$ does not contain any diagonal automorphisms. By Lemma \ref{nondiagaut}, $A$ and $B$ each normalize a conjugate of a fixed parabolic maximal subgroup $P<S$ and $P\rtimes \langle g\rangle$ is maximal in $G$.  

 Our second case is when $S=PSL_n(q)$ such that $n$ is composite and $A$ and $B$ are generated by diagonal automorphisms. Then by \cite[Prop. 3.2.2]{burnessgiudici} $A$ and $B$ both fix a $p$-dimensional subspace of the natural module of $PSL_n(q).$ In this case are exactly the stabilizers of $p$-dimensional subspaces by \cite[Section 3.3.3]{wilsonbook}, so they are contained in conjugate parabolic subgroups.

By \cite[Lemma 4.1.5]{kalathesis}, \cite[Corollary 1.3]{gilleguralnick} and \cite[Thm 2.8.7]{carter}, the intersection of two conjugate parabolics is always non-trivial. Hence in both of the above cases there is a path of length at most $3$ between any two prime order subgroups of $G,$ so $\mathrm{diam}(\Delta_G)\leq 3.$

Our third case is $S=\mathrm{PSU}_n(q)$ with $n$ composite and $A$, $B$ generated by diagonal automorphisms of order $p$ dividing $(n,q+1)$. By \cite[Prop. 3.3.3]{burnessgiudici}, the preimage of each of these elements in $\mathrm{GU}_n(q)$ each fixes an orthogonal decomposition of the natural module into non-degenerate subspaces. Hence, each is contained in a maximal subgroup $M$ of $G$ whose preimage in $\mathrm{GU}_n(q)$
preserves an orthogonal decomposition of the natural module into two complementary non-degenerate subspaces of dimension $m$ and $n-m$ for some $1\leq m\leq n/2$ respectively. By \cite[Tables 2.1.C, 2.1.D, Lemma 4.1.1]{KL}, we deduce that if $m<n/2$, then $|M|>|G|^{1/2}$. Therefore, if $A$ and $B$ are generated by elements that fix a decomposition of the natural module into a pair of non-degenerate subspaces of unequal dimension, then the maximal subgroups $M_1\geq A$, $M_2\geq B$ of $G$ fixing these decompositions must have $|G|<|M_1||M_2| = |M_1M_2||M_1\cap M_2| = |G||M_1\cap M_2|$. This implies $|M_1\cap M_2|>1$, so $A$, $B$ are at distance 3 in $\Delta_G$.
%Moreover, if $m=n/2$ (so $n\geq 6$ even), then $|M|> |G|^{1/2}/q^2$.

With these conditions in mind, we use {\sc Magma} \cite{magma} to prove that the diameter of the intersection graph of $\mathrm{PGU}_4(5)$ is 4. 
\end{proof}

\begin{rem}
 Note that the only cases remaining when $S=\mathrm{PSU}_n(q)$ in the proof of Theorem \ref{thm:nonsimple} are those where $g$ is a diagonal automorphism of prime order $p$ whose preimage in $\mathrm{GU}_n(q)$ does not preserve a decomposition of the natural module into two non-degenerate subspaces of unequal dimension. By \cite[Prop. 3.3.3]{burnessgiudici}, this implies that either $n$ is prime, or $n=2p$.
\end{rem}
% Note that for (1),(2) and (3) in Theorem \ref{thm:nonsimple} determining the exact diameter of the intersection graph is still open. 

\section{Intersection graphs of alternating groups with prime degree}
We now turn our attention to the intersection graphs of finite simple alternating groups of prime degree. As mentioned in the introduction, Cs\'ak\'any and Poll\'ak \cite[p. 246]{csakany} proved that $3\leq \mathrm{diam}(\Delta_{A_n}) \leq 4$, 
with
$\mathrm{diam}(\Delta_{A_n})=3$
if $n$ is composite.

As far as we are aware, no example of a finite simple alternating group whose intersection graph has diameter 4 has been discovered to date. We present a family of such groups below.

% \begin{theorem}
% Let $G=A_n$, where $n$ is a prime not equal to 11 or $(q^d-1)/q-1$ for some prime power $q$ and positive integer $d$. The intersection graph of $G$ has diameter 4.
% \end{theorem}
\subsection{Proof of Theorem \ref{thm:An}}
\begin{proof} We begin by noting that $n \neq 5,7,13,17$, since each can be written as $(q^d-1)/(q-1)$ for some prime power $q$ and positive integer $d\geq 2$. We also note that $n=11$ is excluded by assumption (although we do deal with this case later in Lemma \ref{a11}).

Let $n=23.$ In this case an element of order $23$ is contained in two maximal subgroups isomorphic to $M_{23}$. Using {\sf GAP} \cite{GAP4} we verify that subgroups generated by the $23$-cycles \begin{align*}
    g_a & = (1,13,16,4,22,2,8,20,21,6,17,9,19,14,18,11,15,23,12,5,3,7,10)\\
    g_b &= (1, 2, 3, 4, 5, 6, 7, 8, 9, 10, 11, 12, 13, 14, 15, 16, 17, 18, 19, 20, 21, 22, 23)
\end{align*}
and are of distance 4 by confirming that the $M_{23}$ maximal subgroups they are each contained in intersect trivially. Hence  $\mathrm{diam}(\Delta_{A_{23}})=4.$

For the remainder of the proof, we therefore suppose either $n=19$, or $n\geq 29$. 
Let $A = \langle g_a\rangle$ and $B = \langle g_b \rangle$, where $g_a= (1, 2, \dots ,n-2 , n-1 ,  n)$ and $g_b = g_a^{(n-1,n)}= (1, 2, \dots , n-2 , n ,  n-1)$. Every maximal subgroup of $A_n$ containing $g_a$ or $g_b$ must be transitive of prime degree.
Thus, by the classification of transitive groups of prime degree \cite[p.99 Claims (i)-(v)]{dixon}, the only maximal subgroup of $G$ containing $A$ is $N_G(A) \cong n:\frac{n-1}{2}$, and similarly, the only maximal subgroup of $G$ containing $B$ is $N_G(B) \cong n:\frac{n-1}{2}$.  Suppose $1\neq h \in N_G(A)\cap N_G(B)$. Note that $h$ has order dividing $\frac{n-1}{2}$, since $A$ and $B$ are distinct. Moreover, $h$ preserves a partition of $\{1, \dots n\}$ into two parts of size $\frac{n-1}{2}$, and one part of size 1. In particular, $h$ fixes exactly one point. Relabelling via a conjugate of $g_a$ if necessary, we may suppose that $h$ fixes 1, so now $g_a = (1, 2, \dots ,n-2 , n-1 ,  n)$ and $g_b = g_a^{(k,k+1)}$ with $1\leq k \leq n-1$. 
Hence $g_a^h=(1, (2)h,(3)h,\dots,(n)h$ and the $i^{th}$ position of $g_b^h$ is similarly $(i)h$ unless $i=k$ or $k+1,$ in which case it is $(k+1)h$ and $(k)h,$ respectively.

Since $h$ normalises $\langle g_a\rangle $ and $\langle g_b \rangle$, there exist $1<i,j \leq n-1$ such that $g_a^h = g_a^i$ and $g_b^h = g_b^j$. 

%TODO add explicit g_a^i, g_a^j? Say mod n?

We now fall into one of two cases: either $k=1$, or $k\neq 1$.

Suppose $k=1$. Then $(3)h = 2i+1$, so $(3)h = 2i+1 \mod n$. On the other hand, $(3)h  = j+2 \mod n$, so $j \mod n=2i-1 \mod n$. Now consider $(4)h$. We have $(4)h = 3i+1 \mod n$ and also $(4)h =  2j+2 \mod n = 4i \mod n$. Therefore, $i=j=1$ and $h$ centralises both $g_a$ and $g_b$, a contradiction.

Now suppose $k \neq 1$. Let $\ell \in \{1, \dots , n\}\setminus \{1,k, k+1, (k)h^{-1},(k+1)h^{-1}\}$. Then
\[
(\ell)h =  (\ell-1)i+1 \mod n,
\]
while 
\[
(\ell)h =  (\ell-1)j+1 \mod n,
\]
implying that $i=j$. 

Since $g_b=g_a^{(k,k+1)},$ also $g_b^i=(g_a^i)^{(k, k+1)}=(g_a^h)^{(k,k+1)}.$ Without loss of generality we can assume that there is  $\ell'>\ell''>1$ such that $(\ell')h=k$ and $(\ell'')h=k+1.$ Note that unless $i=1,$ $\ell'-\ell''\geq 2,$ so at least one of them is not $k$ or $k+1.$ Now $g_b^h=(1,\dots, (\ell'-1)h,(\ell')h,(\ell'+1)h,\dots, (\ell''-1)h,(\ell'')h,(\ell''+1)h,\dots)$ and also, $g_b^h=(g_a^h)^{(k,k+1)}=1,\dots, (\ell'-1)h,(\ell'')h,(\ell'+1)h,\dots, (\ell''-1)h,(\ell')h,(\ell''+1)h,\dots).$ Hence $(\ell')h=(\ell'')h,$ which is a contradiction.  

%Note that $g_a$ and $g_b$ agree on every element except for those with entries $k$, $k+1$. Similarly, the powers $g_a^i$ and $g_b^i$ also agree on every element except those with entries $k$ or $k+1$. In particular, if $k$ is in position $l'$ in $g_a^i,$ then $k+1$ is in position $l'$ for $g_b^i, $ so there exists $\ell' \in \{1, \dots, n\}$ such that $g_a^i(\ell')=k$ and  $g_b^i(\ell')=k+1$. 
%We now set out to show that $\ell'\not \in \{1,k,k+1\}$.
%First note that $\ell' \neq 1$ by assumption. If $\ell' = k$, then since $n$ is prime,  $A$ acts sharply transitively on $\{1, \dots, n\}$, $g_a^i(k) = k$ implies $i=1$. This is a contradiction since $h$ does not centralise $g_a$. If instead $\ell' = k+1$, then a similar argument for $B$ gives another contradiction. Therefore, $\ell'\not \in \{1,k,k+1\}$. Hence,
%\[
%k = (\ell')g_a^i = (\ell')g_a^h = (\ell')h
%\]
%and also
%\[
%k+1 = (\ell')g_b^i = (\ell')g_b^h = (\ell')h,
%\]
%a contradiction. 
It follows $h$ does not exist, $ N_G(A)\cap N_G(B)=\{1\}$ and therefore $P$ and $Q$ are at distance at least 4 in $\Delta_G$ and so $\mathrm{diam}(\Delta_G) = 4$. \par 
\end{proof}
Now we consider the cases when $n=11$ or $(q^d-1)/q-1$ for some prime $q.$ We first prove that the intersection graph of $A_{11}$ has diameter 3.

\begin{lemma}\label{a11}
$\mathrm{diam}(\Delta_{A_{11}})=3.$
\end{lemma}
\begin{proof}
By Lemma \ref{lem:diam}  and \cite[p.246]{csakany}, we can assume that if there exists a pair of subgroups at distance 4, then they are two cyclic groups $A$, $B$ of order 11. %Every element of order $n$ is contained in a subgroup $H_i$ with stucture as in the following table. 
%\centering
%\begin{tabular}{c|c|c|c|c|c}
%    $n$ & $5$&$7$&$11$&$13$&$23$ \\
 %    $H_i$& 
%\end{tabular}

Every element of order 11 is contained in an $M_{11}$, and $|M_{11}|=7920$.
If $A$ and $B$ are at distance 4, then the $M_{11}s$ (call them $H_1,H_2$) must meet only in the identity. But $|H_1H_2|\geq |H_1||H_2|/|H_1\cap H_2| = |H_1||H_2| = 7920^2 > |A_{11}|$, a contradiction. So $H_1$ and $H_2$ must intersect non-trivially, and $P$ and $Q$ are at distance 3.
\end{proof}

Finally we show that for $A_n,$ where $n$ is a prime of the form $(q^d-1)/q-1$ for some prime $q,$ there are examples where the intersection graph has diameter $3$ and an example where it is $4$.

\begin{lemma}
    The intersection graphs of $A_5$ and $A_7$ have diameter $3$ and the intersection graph of $A_{13}$ has diameter $4.$
\end{lemma}
\begin{proof}
By By Lemma \ref{lem:diam}  and \cite[p.246]{csakany} we can assume that if there exists a pair of subgroups at distance 4, then they are two cyclic groups $A$, $B$ of order n.
For $n=5$, the size a of a maximal subgroup $M$ containing an $n$-cycle is at least $10$, and since $10^2 > 60$, these all need to intersect, so there are no cyclic subgroups of order $n$ that are distance $4$ apart from each other. Similarly for $n=7$, if $M$  is a maximal subgroup containing an $n$-cycle then $\vert M \vert \geq 120$ and $120^2> \frac{7!}{2}$ so the diameter is again $3$.  
For $n=13,$ we show using {\sf GAP} \cite{GAP4}  that the subgroups generated by $g_a=(1,8,10,13,7,5,6,12,9,11,3,4,2)$ and $g_b=(1,2,3,4,5,6,7,8,9,10,11,12,13)$ are of distance $4$ away from each other by intersecting the maximal subgroups they are contained in, so the diameter of the intersection graph is $4.$
% gap> g;
% (1,2,3,4,5,6,7,8,9,10,11,12,13)
\end{proof}

The question of the diameter of the intersection graph of $A_n$, with $n=(q^d-1)/(q-1)$ with $q$ a prime power and $d\geq 2$ remains open in general.

% For n=23, Check directly that  g_1 = (1,2,3,\dots, 23) and $g_2 = (1,13,16,4,22,2,8,20,21,6,17,9,19,14,18,11,15,23,12,5,3,7,10)$ are at distance 4. 
% For n=5, the size a of a maximal subgroup $M$ contraining an $n$-cycle is at least 10, and since 10^2 > 60, they all need to intersect, so the diameter is 3. Similarly for n=7, $\vert M \vert \geq 120$ and $120^2> \frac{7!}{2}$ so the diameter is again 3.

% n==13 gap> g2;
% (1,8,10,13,7,5,6,12,9,11,3,4,2)
% gap> g;
% (1,2,3,4,5,6,7,8,9,10,11,12,13)

\bibliographystyle{plain}
\bibliography{refs.bib}

\end{document}